\documentclass[12pt,a4paper]{article}
\usepackage{latexsym,hyperref,amssymb,amsfonts,amsmath,amsthm,nccmath,bm,enumitem}
\usepackage[usenames,dvipsnames]{xcolor}

\usepackage[margin=1.5cm]{geometry}

\setlist{topsep=3pt,partopsep=0pt,itemsep=1pt,parsep=0pt}

\hypersetup{
    colorlinks,
    linkcolor={red!80!black},
    citecolor={blue!80!black},
    urlcolor={blue!80!black}
}

\numberwithin{equation}{section}
\newtheorem{theorem}{Theorem}

\newtheorem{lemma}[theorem]{Lemma}
\newtheorem{conjecture}[theorem]{Conjecture}

\theoremstyle{definition}
\newtheorem{definition}[theorem]{Definition}
\newtheorem{construction}[theorem]{Construction}

\renewcommand{\geq}{\geqslant}
\renewcommand{\leq}{\leqslant}
\renewcommand{\ge}{\geqslant}
\renewcommand{\le}{\leqslant}

\def\lref#1{Lemma~$\ref{#1}$}
\def\tref#1{Theorem~$\ref{#1}$}

\def \Cay {{\rm Cay}}

\def \Z {\mathbb{Z}}
\def \B {\mathcal{B}}

\def \mod#1{{\:({\rm mod}\ #1)}}
\def \wfrac#1#2{(#1)/#2}
\def \wwfrac#1#2{#1/#2}

\let\oldproofname=\proofname
\renewcommand{\proofname}{\rm\bf{\oldproofname}}

\author{Darryn Bryant\thanks{School of Mathematics and Physics, The University of Queensland, Qld,
Australia}, Charles J. Colbourn\thanks{School of Computing, Informatics, and Decision Systems Engineering, Arizona State University, Tempe AZ, USA}, Daniel Horsley\thanks{School of Mathematical Sciences, Monash University, Vic, Australia}, Ian M. Wanless\footnotemark[3]}

\title{\bf Steiner triple systems with high chromatic index}
\date{}

\begin{document}
\def\baselinestretch{1.2}\small\normalsize
\maketitle

\begin{abstract}
It has been conjectured that every Steiner triple system of order $v \neq 7$ has chromatic index at most $\wfrac{v+3}{2}$ when $v \equiv 3 \mod{6}$ and at most $\wfrac{v+5}{2}$ when $v \equiv 1 \mod{6}$. Herein, we construct a Steiner triple system of order $v$ with chromatic index at least $\wfrac{v+3}{2}$ for each integer $v \equiv 3 \mod{6}$ such that $v \geq 15$, with four possible exceptions. We further show that the maximum number of disjoint parallel classes in the systems constructed is sublinear in $v$. Finally, we establish for each order $v \equiv 15 \mod{18}$ that there are at least $v^{v^2(1/6+o(1))}$ non-isomorphic Steiner triple systems with chromatic index at least $\wfrac{v+3}{2}$ and that some of these systems are cyclic.
\end{abstract}

\section{Introduction}

A \emph{Steiner triple system} of order $v$ is a pair $(V,\B)$ where $V$ is a $v$-set of \emph{points} and $\B$ is a collection of $3$-subsets of $V$, called {\em triples}, such that any two points occur together in exactly one triple. Steiner triple systems are fundamental in design theory. Kirkman \cite{Kir} proved in 1847 that there exists a Steiner triple system of order $v$ if and only if $v\equiv 1$ or $3\mod 6$ (see \cite{ColRos}).

A \emph{partial parallel class} in a Steiner triple system of order $v$ is a subset of its triples that are pairwise disjoint, a \emph{parallel class} is a partial parallel class with $\wwfrac{v}{3}$ triples (when $v \equiv 3 \mod{6}$), and an \emph{almost parallel class} is a partial parallel class with $\wfrac{v-1}{3}$ triples (when $v \equiv 1 \mod{6}$). The \emph{chromatic index} of a Steiner triple system $(V,\B)$ is the smallest number of
partial parallel classes into which $\B$ can be partitioned. Equivalently, this is the smallest number of colours needed to colour the triples in $\B$ so that no two triples that share a point are assigned the same colour. By observing that any partial parallel class in such a system contains at most $\lfloor\wwfrac{v}{3}\rfloor$ triples, it can be seen that the chromatic index of any Steiner triple system of order $v$ is at least
$$m(v)=
\left\{
  \begin{array}{ll}
    \tfrac{v-1}{2} & \hbox{if $v \equiv 3 \mod{6}$,} \\[3pt]
    \tfrac{v+1}{2} & \hbox{if $v \equiv 1 \mod{6}$.}
  \end{array}
\right.$$
Rosa \cite[p.30]{Ros} has conjectured that the chromatic index of any Steiner triple system of order $v \neq 7$ is in $\{m(v),m(v)+1,m(v)+2\}$. A Steiner triple system of order $v$ with chromatic index $m(v)$ is known \cite{RayWil,VanEtAl} to exist for each $v \equiv 1,3 \mod{6}$ except $7$ and $13$. A result of Pippenger and Spencer \cite{PipSpe} on edge colourings of hypergraphs implies that the maximum chromatic index over the Steiner triple systems of order $v$ is $m(v)+o(v)$ as $v\rightarrow\infty$. Sparse families of Steiner triple systems without parallel classes or almost parallel classes have been constructed \cite{BryHorNoAPC,BryHorNoPC,RosCol}, and any such system of order $v$ has chromatic index at least $m(v)+2$. Aside from this, little seems to be known about Steiner triple systems with few parallel classes or high chromatic index.

In this paper we prove four main results concerning Steiner triple systems of orders congruent to 3 modulo 6 that have few parallel classes or high chromatic index. A simple counting argument shows that a Steiner triple system of order $v \equiv 3 \mod{6}$ with chromatic index at most $\wfrac{v+1}{2}$ must have at least $\wfrac{v+3}{6}$ disjoint parallel classes.

Firstly, we construct a Steiner triple system of order $v$ with chromatic index at least $m(v)+2=\wfrac{v+3}{2}$ for each positive integer $v \equiv 3 \mod{6}$ such that $v \geq 15$, with four possible exceptions (the unique systems of order $3$ and $9$ each have chromatic index $m(v)$). We show that the systems we construct have chromatic index at least $\wfrac{v+3}{2}$ by demonstrating that they contain at most $\wfrac{v-3}{6}$ disjoint parallel classes. In contrast to the $v \equiv 3 \mod{6}$ case, to show that a Steiner triple system of order $v \equiv 1 \mod{6}$ has chromatic index at least $m(v)+2$ by proving an upper bound on the number of disjoint almost parallel classes would require showing that the number of such classes is at most one. Few examples of Steiner triple systems with this property are known and constructing them seems to be a difficult task.

\begin{theorem}\label{mainthm}
For all positive $v\equiv 3\mod 6$ there is a Steiner triple system of
order $v$ having chromatic index at least $\wfrac{v+3}{2}$ except when
$v\in\{3,9\}$ and except possibly when $v\in \{45,75,129,513\}$.
\end{theorem}

For large $v$, the systems we construct have many fewer disjoint parallel classes than we require to prove Theorem \ref{mainthm}. In fact, we can show that they have at most $o(v)$ disjoint parallel classes. This compares favourably with the best known analogous result for Latin squares, which gives a linear bound \cite{EW12,WZ13}.

\begin{theorem}\label{PCBoundThm}
For $v\rightarrow\infty$ with $v \equiv 3 \mod{6}$, there is a Steiner
triple system of order $v$ that contains at most $o(v)$ disjoint
parallel classes.
\end{theorem}

By noting that any Steiner triple system of order $v\equiv15\mod{18}$
produced via the Bose construction \cite{Bos} (see
\cite[p.25]{ColRos}) has chromatic index at least $\wfrac{v+3}{2}$, we
are able to obtain an asymptotic lower bound on the number of such
systems. This bound agrees to within the error term with the asymptotics for the total number of Steiner triple systems (see \cite[p.72]{ColRos}).

\begin{theorem}\label{enumThm}
For $v\rightarrow\infty$ with $v\equiv15\mod{18}$, there are $v^{v^2(1/6+o(1))}$ non-isomorphic Steiner triple systems of order $v$ that have chromatic index at least $\wfrac{v+3}{2}$.
\end{theorem}

Finally, we note that some of these systems are cyclic.

\begin{theorem}\label{t:cyclic}
For each $v\equiv15\mod{18}$ there is a cyclic Steiner triple system of
order $v$ that has chromatic index at least $\wfrac{v+3}{2}$.
\end{theorem}

We prove Theorems \ref{mainthm} and \ref{PCBoundThm} in Section \ref{existSection} and Theorems \ref{enumThm} and \ref{t:cyclic} in Section \ref{enumSection}. We conclude this section by introducing some notation that we will use through the rest of the paper. Let $(\Z_n,+,\cdot)$ be the ring of integers modulo $n$. For a graph $G$ whose vertex set is a subset of $\Z_n$, define the \emph{weight} of an edge $\{x,y\}$ in $G$ to be $x+y$. Let $\Z^*_n$ denote the multiplicative group of units modulo $n$ and, for $a_1,\ldots,a_t \in \Z^*_n$, let $\langle a_1,\ldots,a_t \rangle_n$ denote the subgroup of $\Z^*_n$ generated by $a_1,\ldots,a_t$. For a subset $S$ of $\Z^*_n \setminus \{1\}$, let $\Cay(\Z^*_d,S)$ denote the (simple) graph with vertex set $\Z^*_d$ and an edge between $x$ and $y$ if and only if $xy^{-1}\in S$ or $yx^{-1}\in S$ (note that $\Cay(\Z^*_d,S)$ equals $\Cay(\Z^*_d,S\cup S^{-1})$). For a positive integer $n$, let $D_n$ denote the set of all divisors of $n$ greater than 1. Finally, let $\phi$ denote Euler's totient function.

\section{Proof of Theorems \texorpdfstring{\ref{mainthm}}{1} and \texorpdfstring{\ref{PCBoundThm}}{2}}\label{existSection}

The Steiner triple systems we employ to prove Theorems \ref{mainthm} and \ref{PCBoundThm} are produced by the following construction, which is a generalisation of one used by Schreiber \cite{Sch} and Wilson \cite{Wil}. Their construction is recovered by setting $G_0=\{\{x,-x\}:x \in U\}$ in ours. Another variant of the Wilson-Schreiber construction was used in \cite{BryHorNoPC} to find an infinite family of Steiner triple systems with no parallel classes.

\begin{definition}
For a positive integer $n \equiv 1 \mod{6}$, let $G(n)$ be the graph with vertex set $U=\mathbb{Z}_n \setminus \{0\}$ and edge set $\{\{x,-2x\}:x \in U\} \cup \{\{x,-x\}:x \in U\}$.
\end{definition}

\begin{construction}\label{ModifiedWilSch}
Let $n \equiv 1 \mod{6}$ be a positive integer. We form a Steiner triple system of order $n+2$ as follows.
Let $U=\mathbb{Z}_n \setminus \{0\}$ and let $G=G(n)$. Given a 1-factorisation $\{G_0,G_1,G_2\}$ of $G$, we form the Steiner triple system $(U\cup\{\infty_0,\infty_1,\infty_2\},\mathcal{A}_0 \cup \mathcal{A}_{\infty})$ where
\begin{align*}
\mathcal{A}_0&=\big\{\{a,b,c\}\subseteq U: |\{a,b,c\}|=3, a+b+c=0\big\},\mbox{ and}\\
\mathcal{A}_{\infty}&=\big\{\{x,y,\infty_i\}:\{x,y\}\in E(G_i),i\in\{0,1,2\}\big\} \cup \big\{\{\infty_0,\infty_1,\infty_2\}\big\}.
\end{align*}
\end{construction}

That this construction does indeed produce a Steiner triple system follows easily from the fact that $G$ is the graph whose edges are given by the pairs of elements of $U$ that occur in no triple of $\mathcal{A}_0$. We will show that when used with a judicious choice of 1-factorisation, Construction \ref{ModifiedWilSch} yields Steiner triple systems with few disjoint parallel classes. This fact can then be used to show that most of the systems produced have chromatic index at least $\wfrac{v+3}{2}$.

We shall see that the graph $G(n)$ is a vertex-disjoint union of copies of $\Cay(\Z^*_d,\{-1,-2\})$ for various divisors $d$ of $n$. Accordingly we first establish the existence of $1$-factorisations of $\Cay(\Z^*_d,\{-1,-2\})$ with certain properties. These properties will eventually be used to show that the Steiner triple systems we construct have few disjoint parallel classes.

\begin{lemma}\label{colourCayley}
For each odd integer $d\geq 3$, there is a $1$-factorisation $\{M_0,M_1,M_2\}$ of $\Cay(\Z^*_d,\{-1,-2\})$ such that
\begin{itemize}
\item
if two edges of $\Cay(\Z^*_d,\{-1,-2\})$ have weights $x$ and $-x$ for
some $x \neq 0$, then those two edges are in the same $1$-factor;
\item
if $|\langle -1,-2 \rangle_d| \equiv 0 \mod{4}$, then $M_0$ has no
edges with nonzero weight and $M_1 \cup M_2$ has no edges with zero
weight; and
\item
if $|\langle -1,-2 \rangle_d| \equiv 2 \mod{4}$, then $M_0$ has
exactly $\wwfrac{2\phi(d)}{|\langle -1,-2 \rangle_d|}$ edges with
nonzero weight and $M_1 \cup M_2$ has exactly
$\wwfrac{2\phi(d)}{|\langle -1,-2 \rangle_d|}$ edges with zero weight.
\end{itemize}
\end{lemma}

\begin{proof}
Let $X=\langle -1,-2 \rangle_d$ and for each integer $i$ let $x_i=(-2)^i\in X$.
It can be seen that $|X|$ is even by noting that negation is a fixed-point-free involution on $X$. Let $s=\wwfrac{|X|}2$, and
consider the component $H$ of $\Cay(\Z^*_d,\{-1,-2\})$
with vertex set $X$.
If $-1 \in \langle-2\rangle_d$, then $x_s=-x_0$ and $H$ is the union of the cycle
$(x_0, \ldots, x_{s-1},-x_0,\ldots, -x_{s-1})$ and the matching with edge set $\{\{x_i,-x_i\}: i\in\{0,\ldots,s-1\}\}$. Otherwise,
$-1 \notin \langle-2\rangle_d$, $x_s=x_0$ and $H$ is the union of the cycles $(x_0,\ldots,x_{s-1})$ and $(-x_0,\ldots,-x_{s-1})$ and the matching with edge set
$\{\{x_i,-x_i\}:i\in\{0,\ldots,s-1\}\}$.

If $|X|\equiv 0\mod 4$, then we define a 1-factorisation $\{H_0, H_1, H_2\}$ of $H$ by
\begin{align*}
    E(H_0) &= \big\{\{x_i,-x_i\}:i\in\{0,\ldots,s-1\}\big\}; \\
    E(H_1) &= \big\{\{x_{2i},x_{2i+1}\},\{-x_{2i},-x_{2i+1}\}:i\in\{0,\ldots,\tfrac{s}{2}-1\}\big\};\\
    E(H_2) &= \big\{\{x_{2i+1},x_{2i+2}\},\{-x_{2i+1},-x_{2i+2}\}:i\in\{0,\ldots,\tfrac{s}{2}-1\}\big\}.
\end{align*}
Otherwise, $|X|\equiv 2\mod 4$ and we define $\{H_0, H_1, H_2\}$ by
\begin{align*}
    E(H_0) &= \big\{\{x_i,-x_i\}:i\in\{1,\ldots,s-2\}\big\} \cup \big\{\{x_{s-1},x_{s}\},\{-x_{s-1},-x_{s}\}\big\}; \\
    E(H_1) &= \big\{\{x_{2i},x_{2i+1}\},\{-x_{2i},-x_{2i+1}\}:i\in\{0,\ldots,\tfrac{s-3}{2}\}\big\} \cup \big\{\{x_{s-1},-x_{s-1}\}\big\};\\
    E(H_2) &= \big\{\{x_{2i+1},x_{2i+2}\},\{-x_{2i+1},-x_{2i+2}\}:i\in\{0,\ldots,\tfrac{s-3}{2}\}\big\} \cup \big\{\{x_0,-x_0\}\big\}.
\end{align*}

For $a \in\Z^*_d$ and a graph $H'$ with $V(H') \subseteq \mathbb{Z}^*_d$, denote by $aH'$ the graph with vertex set $\{ax:x \in V(H')\}$ and edge set $\{\{ax,ay\}:\{x,y\} \in E(H')\}$. Let $k=\wwfrac{\phi(d)}{|X|}$ and let $a_1,a_2,\ldots,a_k$ be representatives for the cosets of $X$ in $\Z_d^*$. Note that $\Cay(\Z^*_d,\{-1,-2\}) = a_1H\cup \cdots\cup a_{k}H$ and that the nonzero edge weights of $a_iH$ are in $a_iX$ for $i \in \{1,\ldots,k\}$. For $i\in\{0,1,2\}$,  let $M_i=a_1H_i\cup a_2H_i\cup\cdots\cup a_{k}H_i$. It is routine to check that
$\{M_0,M_1,M_2\}$ is a $1$-factorisation of $\Cay(\Z^*_d,\{-1,-2\})$ with the required properties.
\end{proof}

In Lemma \ref{colourLeave} we show that by combining $1$-factorisations of subgraphs of $G(n)$ given by Lemma \ref{colourCayley} we can produce a $1$-factorisation of $G(n)$, again with desirable properties. We state these properties in terms of the function $f$ defined as follows.

\begin{definition}\label{fDef}
For each odd integer $n\geq 3$, define $f(n)=\sum_{d \in D_n}g(d)$ where
\[g(d)=\left\{
  \begin{array}{ll}
    0 & \hbox{if $|\langle -1,-2 \rangle_d| \equiv 0 \mod{4}$}, \\
    \mfrac{\phi(d)}{|\langle -1,-2 \rangle_d|} & \hbox{if $|\langle -1,-2 \rangle_d| \equiv 2 \mod{4}$}.
  \end{array}
\right. \qquad\quad \]
\end{definition}

\begin{lemma}\label{colourLeave}
Let $n \equiv 1 \mod{6}$ and let $G=G(n)$.
There is a $1$-factorisation $\{G_0,G_1,G_2\}$ of $G$ such that
\begin{itemize}
    \item
if two edges of $G$ have weights $x$ and $-x$ for some $x \neq 0$, then those two
edges are in the same $1$-factor;
    \item
$G_0$ has exactly $2f(n)$ edges with nonzero weight and $G_1 \cup G_2$ has exactly
$2f(n)$ edges with zero weight.
\end{itemize}
\end{lemma}

\begin{proof}
For each $d \in D_n$, let $U^d$ be the set of elements of $\Z_n$ having order $d$ in $(\Z_n,+)$, and let $G^d$ be the subgraph of $G$ induced by $U^d$.
It can be seen that $\{U^d:d\in D_n\}$ is a partition of $\Z_n \setminus \{0\}$, that $G$ is the vertex-disjoint union of the graphs in $\{G^d:d\in D_n\}$, and that the nonzero edge weights of $G^d$ are in $U^d$ for $d \in D_n$.

For each $d \in D_n$, let $\theta_d: \Z_d \rightarrow U^d$ be given by $\theta_d(x)
= (\wwfrac{n}{d})x$. It is clear for each $d \in D_n$ that the map $\theta_d$ is an isomorphism from $\Cay(\Z^*_d,\{-1,-2\})$ to $G^d$. For each $d \in D_n$ and $i\in\{0,1,2\}$, let $G_i^d=\theta_d(M_i)$
where $\{M_0,M_1,M_2\}$ is the $1$-factorisation of $\Cay(\Z^*_d,\{-1,-2\})$ given by \lref{colourCayley}. It is now straightforward to verify that if we let $G_i=\bigcup_{d \in D_n}G_i^d$ for $i\in\{0,1,2\}$, then $\{G_0,G_1,G_2\}$ is the required $1$-factorisation of $G$.
\end{proof}

We can now establish an upper bound on the number of disjoint parallel classes in a Steiner triple system produced by applying Construction $\ref{ModifiedWilSch}$ using a $1$-factorisation given by Lemma \ref{colourLeave}.

\begin{lemma}\label{PCsBound}
For each $n \equiv 1 \mod{6}$ there is a Steiner triple system of order $n+2$ that contains at most $3f(n)+1$ disjoint parallel classes.
\end{lemma}

\begin{proof}
Let $U=\mathbb{Z}_n \setminus \{0\}$ and let $G=G(n)$. Let $(U \cup \{\infty_0,\infty_1,\infty_2\},\mathcal{A})$ be the Steiner triple system of order $n+2$ produced by applying Construction \ref{ModifiedWilSch} using the $1$-factorisation $\{G_0,G_1,G_2\}$ of $G$ given by Lemma \ref{colourLeave}. We will show that this system contains at most $3f(n)+1$ disjoint parallel classes. As in the notation of Construction \ref{ModifiedWilSch}, let $\mathcal{A}_{\infty}$ be the set of triples in $\mathcal{A}$ that contain at least one of $\infty_0,\infty_1,\infty_2$.

Let $\mathcal{P}$ be a set of disjoint parallel classes of $(U \cup \{\infty_0,\infty_1\infty_2\},\mathcal{A})$. Since the sum of the elements in $U$ is 0 and
the sum of any triple in $\mathcal{A} \setminus \mathcal{A}_{\infty}$ is 0,
for every parallel class $P \in \mathcal{P}$ we have
\begin{itemize}
    \item[(i)]
$P$ contains $\{\infty_0,\infty_1,\infty_2\}$; or
    \item[(ii)]
$P$ contains a triple $\{\infty_0,x_0,y_0\}$ where $x_0+y_0 \neq 0$; or
    \item[(iii)]
$P$ contains triples $\{\infty_0,x_0,y_0\}$, $\{\infty_1,x_1,y_1\}$ and $\{\infty_2,x_2,y_2\}$
where $x_0+y_0=0$ and $(x_1+y_1)+(x_2+y_2)=0$.
\end{itemize}
At most one parallel class of $\mathcal{P}$ satisfies (i). Since $G_0$ contains exactly
$2f(n)$ edges with nonzero weight, at most $2f(n)$ parallel classes of $\mathcal{P}$ satisfy (ii).
Since any two edges of $G$ having weights $x$ and $-x$ for some $x\neq 0$ are both in
$G_i$ for some $i\in\{0,1,2\}$, if a $1$-factor $P \in \mathcal{P}$ satisfies (iii), then it
must contain triples $\{\infty_1,x_1,y_1\}$ and $\{\infty_2,x_2,y_2\}$ such that $x_1+y_1=0$ and
$x_2+y_2=0$. Thus, since $G_1 \cup G_2$ has exactly $2f(n)$ edges of zero weight, at most
$f(n)$ parallel classes of $\mathcal{P}$ satisfy (iii). Combining these facts, we have that
$\mathcal{P}$ contains at most $3f(n)+1$ parallel classes.
\end{proof}

In view of Lemma \ref{PCsBound}, to prove Theorem \ref{PCBoundThm} it suffices to show that $f(n)$ is sublinear.

\begin{lemma}\label{l:fbound}
As $n\rightarrow\infty$, $f(n)=O\left(\mfrac{n\log\log n}{\log n}\right)$.
\end{lemma}

\begin{proof}
Let $g$ be as in Definition \ref{fDef} and let $m=\lfloor n^{1/3} \rfloor$. Since $\phi(n)\le n$ and
$|\langle -1,-2 \rangle_n|\ge |\langle 2 \rangle_n|
\ge \log_2 n \ge \log n$, we see that
$g(n)\le \wwfrac{n}{\log n}$ for all $n$. Hence
\[
f(n)=\sum_{d\in D_n}g(d)\le\sum_{i=2}^{m}g(i)
+\sum_{{d>m \atop d\in D_n}} g(d)
\le\sum_{i=2}^{m}i
+\sum_{{d>m \atop d\in D_n}} \mfrac{d}{\log(n^{1/3})}
\le O(n^{2/3})+\mfrac{\sigma(n)}{(\log n)/3},
\]
where $\sigma(n)$ is the sum of the divisors of $n$.
The claimed result now follows from the fact that
$\sigma(n)=O(n\log\log n)$ (see, for example, \cite{MSC}).
\end{proof}

\begin{proof}[\textbf{\textup{Proof of Theorem \ref{PCBoundThm}}}]
This follows directly from Lemmas \ref{PCsBound} and \ref{l:fbound}.
\end{proof}

The following technical lemma will be used in the proof of Theorem
\ref{mainthm} to determine the values of $n$ for which the bound in
Lemma \ref{PCsBound} is strong enough for our requirements.

\begin{lemma}\label{numbercrunch}
For each positive integer $n\equiv 1,5\mod 6$ with
$n\notin\{1,7,11,19,31,43,73,127,511\}$ we have
$3f(n)+1<\wfrac{n+5}6$.
\end{lemma}

\begin{proof}
In this proof we exclusively consider $n\equiv 1,5\mod 6$ with $n>1$. Note that any divisor of such an $n$ is itself congruent to $1,5\mod 6$.
Define $\psi(n)=\phi(n)-18g(n)$ and $\psi^*(n)=\sum_{d \in D_n}\psi(d)$. It follows from $\sum_{d \in D_n}\phi(d)=n-1$ and $f(n)=\sum_{d \in D_n}g(d)$ that $\psi^*(n)=n-1-18f(n)$. Thus, if $\psi^*(n)>0$, then $3f(n)+1<\wfrac{n+5}6$, which is what we need to prove.

For $n<2^{19}$, it is easily verified by computer that
$\psi(n)\geq 0$, except that
\begin{equation}\label{e:negvals}
\psi(7)=-12,\quad
\psi(11)=-8,\quad
\psi(31)=-24,\quad
\psi(43)=-12,\quad
\psi(127)=-36,
\end{equation}
and $\psi^*(n)>0$ except when
$n\in\{7,11,19,31,43,73,127,511\}$.
Thus, we may assume that $n\geq2^{19}$.

Note that $|\langle -1,-2 \rangle_n| > \log_2(n)$.
Let $\xi=17\cdot 13\cdot 11\cdot 7\cdot 5\cdot 3\cdot 2$.
It is known (see \cite{MasShi}) that $\phi(x) > \phi(\xi)=92160$
for any integer $x >\xi$. Thus, since $n\geq2^{19}>\xi$, we have
$$\psi(n) \geq \phi(n)-\mfrac{18\phi(n)}{|\langle -1,-2 \rangle_n|} >
\phi(n)\left(1-\mfrac{18}{19}\right)=\mfrac{\phi(n)}{19} \geq \mfrac{92160}{19}
> 4850.$$
This means that ($\ref{e:negvals}$) displays the only values of $n$ for which
$\psi(n)<0$.
Hence $\psi^*(n)=\sum_{d \in D_n} \psi(d)\ge\psi(n)-12-8-24-12-36>4850-92>0$
for all $n\ge 2^{19}$.
\end{proof}

When $v=33$, we can find a Steiner triple system of order $v$ with
chromatic index at least $\wfrac{v+3}{2}$ in a different manner.

\begin{lemma}\label{33Lemma}
There is a cyclic Steiner triple system of order $33$ that has at most
$5$ disjoint parallel classes and has chromatic index $18$.
\end{lemma}

\begin{proof}
Let
\begin{align*}
\B^*&=\big\{\{i,11+i,22+i\}: 0 \leq i \leq 10\big\},\text{ and}\\
\B'&=\big\{\{i,x+i,y+i\}: (x,y) \in \{(3,7),(5,17),(13,15),(8,14),(9,10)\}, i \in \Z_{33}\big\}.
\end{align*}
It is routine to check that $(\Z_{33},\B^* \cup \B')$ is a Steiner triple system. The sum of the points in each triple in $\B^*$ is congruent to 0 modulo 3 and the sum of the points in each triple in $\B'$ is congruent to 1 modulo 3. Furthermore, the sum of all the points in $\Z_{33}$ is 0. So, because the $11$ triples in a parallel class of $(\Z_{33},\B^* \cup \B')$ partition the points in $\Z_{33}$, it follows that any parallel class must contain at least two triples in $\B^*$. Since there are $11$ triples in $\B^*$, the claim about parallel classes follows. With at most 5 parallel classes in any colouring, at least $18$ colours
will be required. We now describe a colouring with that many colours.
We take $\B^*$ as one colour class, as well as 11 colour classes obtained
by developing
{\footnotesize
\begin{align*}
\big\{
&\{0, 23, 32\}, \{1, 13, 29\}, \{2, 5, 9\}, \{3, 8, 20\}, \{4, 12, 18\}, \{7, 16, 17\},\{10, 15, 27\}, \{11, 24, 26\}, \{14, 22, 28\}, \{21, 30, 31\}\big\}
\end{align*}}
under the map $x\mapsto x+3\mod{33}$. Finally, we have the following
6 colour classes:
{\footnotesize
\begin{align*}
&\big\{\{0,8,14\},\{1,5,31\},\{2,21,29\},\{4,6,24\},\{7,9,27\},\{10,13,17\},\{12,20,26\},\{15,18,22\},\{16,19,23\},\{25,28,32\}\big\},
\\
&\big\{\{0,3,7\},\{1,27,30\},\{4,17,19\},\{6,14,20\},\{8,10,28\},\{9,12,16\},\{11,13,31\},\{18,26,32\},\{22,25,29\}\big\},
\\
&\big\{\{0,4,30\},\{1,14,16\},\{3,6,10\},\{7,20,22\},\{9,17,23\},\{12,15,19\},\{13,26,28\},\{18,21,25\},\{24,27,31\}\big\},
\\
&\big\{\{0,13,15\},\{1,4,8\},\{2,28,31\},\{3,16,18\},\{5,11,30\},\{6,19,21\},\{7,10,14\},\{9,22,24\},\{12,25,27\}\big\},
\\
&\big\{\{0,18,31\},\{1,19,32\},\{2,4,22\},\{3,11,17\},\{5,7,25\},\{10,12,30\},\{13,16,20\},\{15,23,29\},\{21,24,28\}\big\},
\\
&\big\{\{1,3,21\},\{2,8,27\},\{4,7,11\},\{5,24,32\},\{6,9,13\},\{10,23,25\},\{15,28,30\},\{16,29,31\},\{19,22,26\}\big\}.\qedhere
\end{align*}}
\end{proof}

 We are now in a position to prove Theorem \ref{mainthm}.

\begin{proof}[\textbf{\textup{Proof of Theorem \ref{mainthm}}}]
For $v\in\{3,9\}$ there is a unique Steiner triple system of order $v$, and it has chromatic index $\wfrac{v-1}2$.
Simple counting shows that any Steiner triple system of order $v \equiv 3 \mod{6}$ with chromatic index at most $\wfrac{v+1}{2}$ must have at least $\wfrac{v+3}{6}$ disjoint parallel classes. So it suffices to show that for each positive $v \equiv 3 \mod{6}$ with $v \notin \{3,9,45,75,129,513\}$, there is a Steiner triple system of order $v$ with at most $\wfrac{v-3}{6}$ disjoint parallel classes. It is known that there are Steiner triple systems of order $21$ with no parallel classes \cite{MatRos}, and by Lemma \ref{33Lemma} there is a Steiner triple system of order 33 with at most $5$ disjoint parallel classes.  Thus, by Lemma \ref{PCsBound}, it suffices to show that for all $v \equiv 3 \mod{6}$ with $v \notin \{3,9,21,33,45,75,129,513\}$ we have $3f(v-2)+1<\wfrac{v+3}{6}$. This follows by applying Lemma \ref{numbercrunch} with $n=v-2$.
\end{proof}

For each order $v \equiv 3 \mod{6}$ with $v\le63$ we built an STS($v$)
by applying Construction~$\ref{ModifiedWilSch}$ using a
$1$-factorisation given by \lref{colourLeave}. We confirmed that each
of these systems achieved the bound in \lref{PCsBound} in that they
possessed $3f(v-2)+1$ disjoint parallel classes. Moreover, the systems
for $v\in\{3,9\}$ had chromatic index $(v-1)/2$ and those for
$v\in\{21,33,45\}$ had chromatic index $(v+1)/2$, while for
$v\in\{15,27,39,51,57,63\}$ they had chromatic index
$(v+3)/2$. Colourings that prove these claims can be downloaded from
\cite{WWWW}.

\section{Proof of Theorems \texorpdfstring{\ref{enumThm}}{3} and \texorpdfstring{\ref{t:cyclic}}{4}}\label{enumSection}

In this section we prove Theorems \ref{enumThm} and \ref{t:cyclic}. We do this by first establishing that any Steiner triple system of order $v\equiv15\mod{18}$ produced via the well-known Bose construction (see \cite[p.25]{ColRos}) has chromatic index at least $\wfrac{v+3}{2}$, and then establishing a lower bound on the number of nonisomorphic such systems. We formalise the relevant instance of the Bose construction as Construction~\ref{Const:Bose}.

For our purposes, a \emph{Latin square on a set $X$} is an array $L$ whose rows and columns are indexed by $X$ and whose cells each contain a symbol from $X$ such that each symbol occurs once in each row and each column. For $x,y \in X$ we denote the symbol in the $(x,y)$ cell of $L$ by $L(x,y)$. We say that $L$ is \emph{idempotent} if $L(x,x)=x$ for each $x \in X$ and \emph{symmetric} if $L(x,y)=L(y,x)$ for all $x,y \in X$. The \emph{order} of $L$ is $|X|$.

\begin{construction}\label{Const:Bose}
Let $k$ be a nonnegative integer. Given an ordered triple $(L_0,L_1,L_2)$ of idempotent symmetric Latin squares, each on a common set $X$ of $6k+5$ elements, we form the Steiner triple system $(X \times \Z_3,\B^* \cup \B_0 \cup \B_1 \cup \B_2)$ where
\begin{align*}
\B^*&=\big\{\{(x,0),(x,1),(x,2)\}:x \in X\big\},\text{ and}\\
\B_i&=\big\{\{(x,i),(y,i),(L_i(x,y),i+1)\}:x,y \in X,x\neq y\big\}
\mbox{ for }i \in \Z_3.
\end{align*}
\end{construction}

A short argument concerning the second coordinates of the points in the Steiner triple systems produced via Construction \ref{Const:Bose} shows that these systems have few parallel classes and hence sufficiently high chromatic index.

\begin{lemma}\label{BoseLemma}
Any Steiner triple system of order $v=18k+15$ produced by
Construction~$\ref{Const:Bose}$ has at most $3k+2$ disjoint parallel
classes and hence has chromatic index at least $\wfrac{v+3}{2}$.
\end{lemma}

\begin{proof}
Let $(X \times \Z_3,\B)$ be a Steiner triple system produced by Construction \ref{Const:Bose}. Simple counting shows that any Steiner triple system of order $v \equiv 3 \mod{6}$ with chromatic index at most $\wfrac{v+1}{2}$ must have at least $\wfrac{v+3}{6}$ disjoint parallel classes, so it suffices to show that $(X \times \Z_3,\B)$ has at most $3k+2$ parallel classes.

As in the notation of Construction \ref{Const:Bose}, let $\B^*$ be the set of triples in $\B$ that contain points with three distinct second coordinates. Note that the sum of the second coordinates in each triple in $\B^*$ is 0 and that the sum of the second coordinates in each triple in $\B \setminus \B^*$ is 1. Furthermore, the sum of the second coordinates of the points in $X \times \Z_3$ is 0. So, because the triples in a parallel class of $(X \times \Z_3,\B)$ partition the points in $X \times \Z_3$, it follows that any parallel class must contain at least two triples in $\B^*$. Since there are $6k+5$ triples in $\B^*$, the result follows.
\end{proof}

We can now prove \tref{enumThm}.

\begin{proof}[\textbf{\textup{Proof of Theorem \ref{enumThm}}}]
Let $n=\wwfrac{v}{3}$. In Construction \ref{Const:Bose}, we choose each of $L_0,L_1,L_2$ independently from the set of all idempotent symmetric Latin squares of order $n$. By \cite[p.66]{Cam76} there are $n^{(1/2+o(1))n^2}$ choices for each Latin square and by \lref{BoseLemma}, the result is always a Steiner triple system of order $v$ that has chromatic index at least $\wfrac{v+3}{2}$. Each Steiner triple system that we generate is generated at most $v!<v^v$ times, up to isomorphism. Therefore the number of non-isomorphic Steiner triple systems that we generate is at least
\[
n^{(3/2+o(1))n^2}/v^v=v^{v^2(1/6+o(1))}.
\]
This agrees to within the error term with the asymptotics for the number of Steiner triple systems (see \cite[p.72]{ColRos}).
\end{proof}

\tref{enumThm} should {\em not} be interpreted as suggesting that most
Steiner triple systems of order $v\equiv3\mod6$ have chromatic index
at least $\wfrac{v+3}{2}$.  In fact, we believe the opposite is
true. We generated 1\,000 random systems for each $v\equiv1,3\mod6$ in
the range $20<v<40$ using a hill climbing algorithm. We do not claim
to have sampled from a uniform distribution but we believe the results
are still indicative of the trend. We found that with a single
exception all the systems that we generated were easy to colour
heuristically with $m(v)+1$ colours. The single exception was an
STS(21) that had chromatic index $m(v)+2$ (it had 3 disjoint parallel
classes but not 4). It is known that neither of the 2 STS(13)'s, 63 of the 80 STS(15)'s \cite{ColRos}, and only 2 of the $11\,084\,874\,829$ STS(19)'s \cite{ColEtAl} have chromatic index $m(v)+2$.  Therefore we propose the following:

\begin{conjecture}
The proportion of Steiner triple systems of order $v$ that have
chromatic index at least $m(v)+2$ tends to $0$ as $v\rightarrow\infty$.
\end{conjecture}

In \lref{33Lemma} we built an STS(33) containing at most $5$ disjoint
parallel classes.  By applying Construction \ref{Const:Bose} with
idempotent symmetric Latin squares of order 11, we can produce
STS(33)'s with similar properties (some of which are isomorphic to the
example in \lref{33Lemma}).  More generally, we can build an infinite
family of cyclic Steiner triple systems with chromatic index at least
$\wfrac{v+3}{2}$.

\begin{proof}[\textbf{\textup{Proof of Theorem \ref{t:cyclic}}}]
Let $n=v/3=6k+5$. Since $n$ is odd, there is Latin square $L$ of order $n$
defined by $L(i,j)=(i+j)/2\mod n$. It is easy to check that $L$
is idempotent and symmetric and has a cyclic automorphism of order $n$.
Applying Construction \ref{Const:Bose} using $(L,L,L)$, we obtain
a Steiner triple system whose automorphism group has a subgroup isomorphic
to $\Z_n\times\Z_3$. As $n$ is coprime to $3$, this means the Steiner
triple system is cyclic. The claim about chromatic index follows
from \lref{BoseLemma}.
\end{proof}

There are 84 cyclic STS(33)'s and 11\,616 cyclic STS(45)'s \cite{ColRos} (also see \cite{Bays,CM80,Kau}).
Of the 84 cyclic STS(33)'s, 5 can be created using Construction~\ref{Const:Bose} and thus have chromatic index at least 18 by \lref{BoseLemma}. By computation, we found that these 5 have chromatic index exactly 18, while the other 79 cyclic STS(33)'s have chromatic index at most 17.  We also confirmed that every cyclic
STS(45) has at least 8 disjoint parallel classes, so they cannot be
used to resolve the first gap in \tref{mainthm} using the techniques
that we have considered.

\vspace{0.3cm} \noindent{\bf Acknowledgements}

We acknowledge the support of the Australian Research Council via grants DP150100530,\linebreak DP150100506, DP120100790 and DE120100040 and the National Science Foundation via grant 1421058. We also thank S.\ Herke for writing the code that generates random Steiner triple systems.

  \let\oldthebibliography=\thebibliography
  \let\endoldthebibliography=\endthebibliography
  \renewenvironment{thebibliography}[1]{%
    \begin{oldthebibliography}{#1}%
      \setlength{\parskip}{0.4ex plus 0.1ex minus 0.1ex}%
      \setlength{\itemsep}{0.4ex plus 0.1ex minus 0.1ex}%
  }%
  {%
    \end{oldthebibliography}%
  }

\end{document}